\documentclass{amsart}

\usepackage{hyperref}
\usepackage{amsmath,amsthm,amssymb}

\newtheorem{thm}{Theorem}
\newtheorem{prop}[thm]{Proposition}
\newtheorem{cor}[thm]{Corollary}
\newtheorem{lem}[thm]{Lemma}
\newtheorem{clm}[thm]{Claim}

\theoremstyle{definition}
\newtheorem{dfn}[thm]{Definition}
\newtheorem{rem}[thm]{Remark}
\newtheorem{ex}[thm]{Example}

\newtheorem{ques}[thm]{Question}

\numberwithin{thm}{section}
\numberwithin{equation}{section}

\DeclareMathOperator{\dist}{dist}
\DeclareMathOperator{\diam}{diam}
\DeclareMathOperator{\Ant}{Ant}
\DeclareMathOperator{\tp}{top}

\newcommand{\step}[1]{\medskip\noindent\textit{#1.}}

\title[Extremal subsets in GCBA spaces]{Extremal subsets in geodesically complete spaces with curvature bounded above}
\author[T. Fujioka]{Tadashi Fujioka}
\address{Department of Mathematics, Osaka University, Toyonaka, Osaka 560-0043, Japan}
\email{fujioka@cr.math.sci.osaka-u.ac.jp, tfujioka210@gmail.com}
\date{\today}
\thanks{Supported by JSPS KAKENHI Grant Numbers 15H05739, 20H00114, 22KJ2099 (22J00100)}

\begin{document}

\begin{abstract}
We introduce the notion of an extremal subset in a geodesically complete space with curvature bounded above, i.e., a GCBA space.
This is an analogue of an extremal subset in an Alexandrov space with curvature bounded below introduced by Perelman and Petrunin.
We prove that under an additional assumption the set of topological singularities in a GCBA space forms an extremal subset.
We also exhibit some structural properties of extremal subsets in GCBA spaces.
\end{abstract}

\maketitle

\section{Introduction}\label{sec:intro}

The notion of an extremal subset, introduced by Perelman and Petrunin \cite{PP:ext}, plays an essential role in the geometry of finite-dimensional Alexandrov spaces with curvature bounded below, or CBB spaces.
Extremal subsets are singular sets of CBB spaces defined by a purely metric condition, but also closely related to the topological structure of the ambient spaces.
For example, the boundary of a CBB space and the set of non-manifold points are extremal subsets (\cite[Examples 1.2]{PP:ext}).
More generally, any CBB space admits a unique maximal stratification into topological manifolds such that the closures of the strata are extremal subsets (\cite[3.8]{PP:ext}).
See also \cite[Section 4]{Pet:semi} for a brief summary.

On the other hand, metric spaces with curvature bounded above, or CBA spaces, do not have such a nice singular structure.
Even if we restrict our attention to GCBA spaces, that is, separable, locally compact, locally geodesically complete CBA spaces, there exists a $2$-dimensional GCBA space with no such topological stratification, which arises as a limit of polyhedral spaces (\cite[p.\ 412]{Kl}, \cite[Example 2.7]{N:asym}).
Although there is a notion of stratification of GCBA spaces defined by Lytchak and Nagano \cite{LN:geod}, it is much more measure-theoretic.

Nevertheless, the results of Lytchak and Nagano \cite{LN:geod}, \cite{LN:top} suggest that there are many similarities between CBB and GCBA geometries.
From this point of view, the author \cite{F:nc} recently introduced the regularity of distance maps on GCBA spaces, which can be viewed as a dual concept of Perelman's regularity in CBB geometry (\cite{Per:alex}, \cite{Per:mor}; originally introduced by Grove and Shiohama \cite{GS} in the Riemannian setting for distance functions).
Since extremal subsets in CBB spaces are defined by using this regularity, one can literally translate it to GCBA spaces.
Although GCBA extremal subsets do not have nice properties as CBB ones, it would be meaningful to exhibit the difference between CBB and GCBA geometries.

\begin{dfn}\label{dfn:ext}
Let $X$ be a GCBA space and $E$ a closed subset.
We say that $E$ is \textit{extremal} if it satisfies the following condition:
\begin{enumerate}
\item[(E)]
any point of $E$ has a neighborhood $U$ in $X$ (contained in some tiny ball) such that if $p\in U\cap E$ is a local minimum point of $\dist_q|_E$, where $q\in U\setminus E$, then $p$ is a critical point of $\dist_q$, i.e., $\min\overline{|q'_p\cdot|}\ge\pi/2$.
\end{enumerate}
\end{dfn}

Here a tiny ball is a convex ball where triangle comparison holds.
The notation $\overline{|\ ,\ |}$ denotes the antipodal distance on the space of directions and $q'_p$ denotes the direction of the unique shortest path from $p$ to $q$.
See Section \ref{sec:pre} for more details.
The condition $\min\overline{|q'_p\cdot|}\ge\pi/2$ means that for any direction $\xi$ at $p$ there exists an antipode of $q'_p$ making an angle $\ge\pi/2$ with $\xi$.
The above definition is almost the same as the CBB one, except for the definition of a critical point: a CBB critical point is defined by $\max|Q'_p\cdot|\le\pi/2$, where $Q'_p$ denotes the set of all directions of shortest paths from $p$ to $q$.

We give some examples (see Section \ref{sec:ex} for details).
Let $\Sigma$ be a compact GCAT($1$) space, i.e., a compact, geodesically complete CAT($1$) space.
Consider the Euclidean cone over the rescaled space $\lambda\Sigma$, where $\lambda\ge3/2$.
Then its vertex is an extremal point, that is, a one-point extremal subset (Proposition \ref{prop:ver}).
Another basic example comes from Reshetnyak's gluing theorem (\cite{R}): one can glue two GCBA spaces along their isometric convex subsets to obtain a new GCBA space.
Under a reasonable assumption, the convex subset, and also its complement and boundary are extremal subsets in the glued space (Proposition \ref{prop:conv'}).
In particular, any convex subset in a GCBA space, and its complement and boundary are extremal subsets in the doubled space (Proposition \ref{prop:conv}).
The union of the measure-theoretic strata of Lytchak-Nagano \cite{LN:geod} is extremal if it is closed (Proposition \ref{prop:str}).

We are especially interested in the relationship between metric singularities and topological singularities.
Recall that any non-manifold point of a CBA homology manifold is isolated (\cite[Theorem 1.2]{LN:top}).
We show that such an isolated topological singularity of a GCBA space is an extremal point (Proposition \ref{prop:pt}).
More generally, we prove the following theorem.

\begin{thm}\label{thm:hom}
Let $X$ be a GCBA space and fix a positive integer $k$.
Let $X^k_{\hom}$ denote the set of points in $X$ whose spaces of directions are homotopy equivalent to $\mathbb S^{k-1}$.
If $X^k_{\hom}$ is open, then its complement is an extremal subset of $X$.
\end{thm}

The assumption of the above theorem has another formulation.

\begin{rem}\label{rem:hom}
Let $X^k_{\tp}$ denote the set of points in $X$ having neighborhoods homeomorphic to $\mathbb R^k$.
Then $X^k_{\tp}\subset X^k_{\hom}$, and the equality holds if and only if $X^k_{\hom}$ is open.
This follows from the topological regularity theorem of Lytchak-Nagano \cite[Theorem 1.1]{LN:top}.
In particular, the assumption of Theorem \ref{thm:hom} is satisfied when both $\dim X$ and $k$ are at most $2$ (see \cite[Proposition 3.1]{N:sph} for $\dim X=k=2$).
\end{rem}

However, when $\dim X\ge 3$, the assumption of Theorem \ref{thm:hom} is not always satisfied.
The author is grateful to Alexander Lytchak for the following example.

\begin{ex}\label{ex:hom}
Let $\Sigma$ be a contractible compact GCAT($1$) space of dimension $\ge 2$.
Such a space exists because there exists a contractible finite simplicial complex without boundary (such as the dunce hat; \cite{Z}) and any finite simplicial complex admits a CAT($1$) metric, due to Berestovskii \cite{B}.
The wedge of $\Sigma$ and $\mathbb S^{k-1}$ is a compact GCAT($1$) space that is homotopy equivalent to $\mathbb S^{k-1}$.
If $X$ is the Euclidean cone over this wedge, then $X^k_{\hom}$ is not open at the vertex.
\end{ex}

Although the assumption of Theorem \ref{thm:hom} does not hold in general, the following question still makes sense.

\begin{ques}
Is the complement of $X^k_{\tp}$ extremal?
This is true for CBB spaces (where $k=\dim X$; see \cite[Example 1.2]{PP:ext}).
Note that $X\setminus X^k_{\tp}$ is equal to the closure of $X\setminus X^k_{\hom}$ by the topological regularity theorem mentioned above. 
\end{ques}

Now we describe the structural properties of extremal subsets in GCBA spaces.
In what follows, $E$ denotes an extremal subset of a GCBA space $X$.
The following generalizes \cite[Theorem 1.1]{LN:geod}.

\begin{thm}\label{thm:dim}
The Hausdorff dimension of any open subset $U$ of $E$ coincides with the topological dimension.
It also coincides with the maximal dimension of open subsets in $U$ that are (bi-Lipschitz) homeomorphic to Euclidean balls.
\end{thm}

This theorem allows us to define the \textit{local dimension} of $E$ at $p\in E$ by
\begin{equation}\label{eq:dim}
\dim_pE:=\lim_{r\to0}\dim B(p,r)\cap E,
\end{equation}
where the right-hand side denotes the Hausdorff dimension.
It is clearly upper semicontinuous in $p$.
For any nonnegative integer $k$, we define the \textit{$k$-dimensional part} $E^k$ of $E$ as the set of points in $E$ with local dimension $k$.
Note that the set $E^0$ of extremal points in $E$ is discrete since any distance function from a point has no critical points around it (see \cite[Proposition 7.3]{LN:geod}).

\begin{rem}\label{rem:dim}
One can also define the \textit{infinitesimal dimension} of $E$ at $p$ by the dimension of the tangent cone $T_pE$ of $E$ at $p$, which is an extremal subset of the tangent cone $T_p$ of $X$ at $p$ (Corollary \ref{cor:cone}).
We always have $\dim_pE\ge\dim T_pE$ (Corollary \ref{cor:dim2}).
However, the equality does not hold in general, except when $E=X$ (\cite[Theorem 1.2]{LN:geod}).
See Remark \ref{rem:dim2} for a counterexample.
\end{rem}

From now on we assume $k\ge 1$.
We say that $p\in E$ is a \textit{$k$-regular point} of $E$ if $T_pE$ is isometric to $\mathbb R^k$ with respect to the metric of $T_p$.
Note that a $k$-regular point is not necessarily contained in $E^k$ (see again Remark \ref{rem:dim2}).

\begin{thm}\label{thm:reg}
The set of $k$-regular points of $E$ is dense in $E^k$ and its complement in $E^k$ has Hausdorff dimension at most $k-1$.
The $k$-dimensional Hausdorff measure is locally finite and locally positive on $E^k$.
\end{thm}

\begin{thm}\label{thm:mfd}
$E^k$ is $k$-rectifiable, i.e., there exists a countable family of Lipschitz maps $f_i:A_i\to E^k$, where $A_i\subset\mathbb R^k$, such that $\bigcup_{i=1}^\infty f_i(A_i)$ has full measure in $E^k$ (with respect to the $k$-dimensional Hausdorff measure).
Moreover, $E^k$ contains a Lipschitz manifold of dimension $k$ that is dense in $E^k$ and open in $E$.
\end{thm}

Compare the above theorems with \cite[Theorems 1.2, 1.3]{LN:geod}.

\begin{rem}
The set of manifold points in $E^k$ does not always have full measure.
See \cite[Example 6.9]{NSY}.
On the other hand, in the CBB case, the set of non-manifold points in $E^k$ has Hausdorff dimension at most $k-1$ (\cite[Theorem 1.1]{F:reg}).
The reason for this difference can be found in Remark \ref{rem:homeo}.
\end{rem}

\begin{rem}\label{rem:prf}
In the proofs of the above theorems, we first show that the Hausdorff dimension of any open subset of an extremal subset is an integer.
This is enough to define the local dimension by \eqref{eq:dim}.
Then we prove Theorems \ref{thm:reg} and \ref{thm:mfd}.
Finally we deduce Theorem \ref{thm:dim} from these two.
\end{rem}

It should be mentioned that there is a notion of wall singularity of GCBA spaces introduced by Nagano \cite{N:wall}.
This notion has some relation with extremal subsets, but we will not treat it in this paper.

\step{Organization}
In Section \ref{sec:pre}, after recalling some basic notions in GCBA geometry, we review the properties of noncritical maps introduced in \cite{F:nc}.
Section \ref{sec:ex} provides several examples of GCBA extremal subsets mentioned above, including the proof of Theorem \ref{thm:hom}.
In Section \ref{sec:prop}, we develop the general theory of GCBA extremal subsets and prove Theorems \ref{thm:dim}, \ref{thm:reg}, and \ref{thm:mfd}.

\step{Acknowledgments}
I would like to thank Professors Koichi Nagano and Alexander Lytchak for their helpful comments.

\section{Preliminaries}\label{sec:pre}

In this paper, we denote by $|pq|$ the distance between $p$ and $q$.
We also denote by $\dist_p:=|p\cdot|$ the distance function from $p$.
The notation $B(p,r)$, $\bar B(p,r)$, and $\partial B(p,r)$ denote the open and closed $r$-ball around $p$ and its boundary, respectively.

\subsection{GCBA spaces}\label{sec:gcba}

In this subsection we recall some basic notions in GCBA geometry.
See \cite{LN:geod}, \cite{LN:top}, \cite{BBI}, and \cite{BH} for more details.

Let $\kappa$ be a real number.
A \textit{CAT($\kappa$) space} is a complete metric space such that any two points with distance $<\pi/\sqrt\kappa$ are joined by a shortest path and any geodesic triangle with perimeter $<2\pi/\sqrt\kappa$ is not thicker than the comparison triangle in the model plane of constant curvature $\kappa$, where $\pi/\sqrt\kappa=\infty$ when $\kappa\le0$.
A \textit{CBA($\kappa$) space} is a locally CAT($\kappa$) space.

A CBA space is called \textit{locally geodesically complete} (resp.\ \textit{geodesically complete}) if any geodesic is locally extendable (resp.\ infinitely extendable), where a geodesic means a locally shortest path.
These two definitions are equivalent if the space is complete.
For example, if any sufficiently small punctured ball in a CBA space is not contractible, then the space is locally geodesically complete.
This assumption is satisfied by a CBA homology manifold (see below).

A \textit{GCBA space} is a separable, locally compact, locally geodesically complete CBA space.
Similarly, in this paper, a \textit{GCAT space} refers to a separable, locally compact, geodesically complete CAT space.

Let $X$ be a GCBA($\kappa$) space and $p\in X$.
We denote by $T_p$ the \textit{tangent cone} at $p$ and by $\Sigma_p$ the \textit{space of directions} at $p$.
Recall that $T_p$ is the Euclidean cone over $\Sigma_p$, and is also isometric to the unique blow-up limit of $X$ at $p$.
In particular, $T_p$ is a GCAT($0$) space and $\Sigma_p$ is a compact GCAT($1$) space.

The \textit{dimension} of $X$ is defined by the Hausdorff dimension, which coincides with the topological dimension.
The \textit{local dimension} of $X$ at $p$ is defined by the formula \eqref{eq:dim} for $E=X$, which is finite and coincides with the dimension of $T_p$.

A metric ball $U$ of radius $r$ in $X$ is called \textit{tiny} if $r<\min\{\pi/100\sqrt\kappa,1\}$ and the closed concentric ball of radius $10r$ is a compact CAT($\kappa$) space.
For two different points $p,q\in U$, we denote by $q'=q'_p\in\Sigma_p$ the direction of the unique shortest path from $p$ to $q$.

For a closed subset $A\subset X$ and $p\in A$, we define the \textit{space of directions} $\Sigma_pA$ of $A$ at $p$ by
\[\Sigma_pA:=\left\{\xi\in\Sigma_p\mid\xi=\lim(q_i)'_p,\ A\setminus\{p\}\ni q_i\to p\right\}.\]
Clearly $\Sigma_pA$ is a closed subset of $\Sigma_p$.
We also define the \textit{tangent cone} $T_pA$ of $A$ at $p$ as the subcone over $\Sigma_pA$ in $T_p$.

Let $\Sigma$ be a compact GCAT($1$) space.
In this paper, we always consider the \textit{$\pi$-truncated metric} on $\Sigma$, that is, the minimum of the original metric and $\pi$.
For the sake of induction, we assume that a $0$-dimensional compact GCAT($1$) space is a finite set of points with pairwise distance $\pi$ that is not a singleton.
Note that any $\xi\in\Sigma$ has at least one \textit{antipode} $\bar\xi\in\Sigma$, i.e., $|\xi\bar\xi|=\pi$.
We denote by $\Ant(\xi)$ the set of all antipodes of $\xi$ in $\Sigma$.

In this paper a \textit{homology manifold} means a separable, locally compact metric space of finite topological dimension whose local homology groups (with integer coefficients) are isomorphic to those of Euclidean space of fixed dimension.
As mentioned above, any CBA homology manifold is locally geodesically complete, and hence is a GCBA space.
Note that since a CBA homology manifold is locally contractible, one can use Poincar\'e duality.

\subsection{Noncritical maps}\label{sec:nc}

In this subsection we review the properties of noncritical maps introduced in \cite{F:nc}.
We first define the antipodal distance.

\begin{dfn}
Let $\Sigma$ be a compact GCAT($1$) space and $\xi,\eta\in\Sigma$.
The \textit{antipodal distance} between $\xi$ and $\eta$ is defined by
\begin{align*}
\overline{|\xi\eta|}:&=\max_{x\in\Sigma}|\xi x|+|\eta x|-\pi\\
&=\max_{\bar\xi\in\Ant(\xi)}|\bar\xi\eta|=\max_{\bar\eta\in\Ant(\eta)}|\xi\bar\eta|.
\end{align*}
The last two equalities follow from the geodesic completeness (see \cite[Lemma 4.1]{F:nc}).
\end{dfn}

Based on Perelman's idea in CBB geometry (\cite{Per:alex}, \cite{Per:mor}), the author \cite{F:nc} introduced the following regularity of distance maps on GCBA spaces.
Suppose $\varepsilon$ and $\delta$ are positive numbers such that $\delta\ll\varepsilon$.
Here the choice of $\delta$ depends only on (local) dimension and $\varepsilon$, which will be determined by the proof of each statement.

\begin{dfn}\label{dfn:nc}
Let $U$ be a tiny ball in a GCBA space and $a_i\in U$ ($1\le i\le k$).
We say $f=(|a_1\cdot|,\dots,|a_k\cdot|):U\to\mathbb R^k$ is \textit{$(\varepsilon,\delta)$-noncritical} at $p\in U\setminus\{a_1,\dots,a_k\}$ if
\begin{enumerate}
\item $\overline{|a_i'a_j'|}<\pi/2+\delta$ at $p$ for any $1\le i\neq j\le k$;
\item there exists $b\in U\setminus\{p\}$ such that $\overline{|a_i'b'|}<\pi/2-\varepsilon$ at $p$ for any $1\le i\le k$.
\end{enumerate}
A map $f$ is called \textit{noncritical} at $p$ if it is $(\varepsilon,\delta)$-noncritical at $p$ for some $\delta\ll\varepsilon$.
Moreover, $f$ is called \textit{noncritical} on $V\subset U$ if it is noncritical at any $p\in V$.
\end{dfn}

In particular, we also say $\dist_q$ ($q\in U\setminus\{p\}$) is \textit{regular} at $p$ if $\min\overline{|q_p'\cdot|}<\pi/2$; otherwise it is \textit{critical} at $p$.
Now Definition \ref{dfn:ext} makes sense.

Note that the noncriticality is an open condition.
This easily follows from the upper semicontinuity of angles and the local geodesic completeness.

To describe the infinitesimal properties of noncritical maps, we introduce

\begin{dfn}\label{dfn:nc'}
Let $\Sigma$ be a compact GCAT(1) space and let $\xi_i\in\Sigma$ ($1\le i\le k$).
We say $\{\xi_i\}_{i=1}^k$ is an \textit{$(\varepsilon,\delta)$-noncritical collection} of $\Sigma$ if
\begin{enumerate}
\item $\overline{|\xi_i\xi_j|}<\pi/2+\delta$ for any $1\le i\neq j\le k$;
\item there exists $\eta\in\Sigma$ such that $\overline{|\xi_i\eta|}<\pi/2-\varepsilon$ for any $1\le i\le k$.
\end{enumerate}
We call $\eta$ a \textit{regular direction} for $\{\xi_i\}_{i=1}^k$.
As in the previous definition, ``$(\varepsilon,\delta)$'' may be omitted.
\end{dfn}

Clearly $f=(|a_1\cdot|,\dots,|a_k\cdot|)$ is an $(\varepsilon,\delta)$-noncritical map at $p$ if and only if $\{a_i'\}_{i=1}^k$ is an $(\varepsilon,\delta)$-noncritical collection of $\Sigma_p$.

\begin{rem}[{\cite[Remark 4.3]{F:nc}}]\label{rem:ineq}
In the situation of Definition \ref{dfn:nc'}, we have $|\xi_i\xi_j|>\pi/2-\delta$ and $|\xi_i\eta|>\pi/2+\varepsilon$.
Moreover, if $k\ge 2$, we also have $|\xi_i\xi_j|<\pi-2\varepsilon$ and $|\xi_i\eta|<\pi-\varepsilon/2$.
\end{rem}

The next lemma is the key property of noncritical collections.
In this paper, $c(\varepsilon)$ denotes various positive constants depending only on (local) dimension and $\varepsilon$.
Similarly, $\varkappa(\delta)$ denotes various positive functions depending only on (local) dimension and $\varepsilon$ such that $\varkappa(\delta)\to0$ as $\delta\to 0$.
Usually $c(\varepsilon)\ll\varepsilon$ and $\varkappa(\delta)\gg\delta$.
Whenever these symbols appear in a statement, it means that there exists such a constant or function for which the statement holds.

In what follows, $\Sigma$ denotes a compact GCAT($1$) space.

\begin{lem}[{\cite[Lemma 4.4]{F:nc}}]\label{lem:ind}
Suppose $\dim\Sigma\ge 1$.
Let $\{\xi_i\}_{i=1}^k$ be an $(\varepsilon,\delta)$-noncritical collection of $\Sigma$ with a regular direction $\eta$.
For $x\in\Sigma$, assume either
\begin{enumerate}
\item $\xi_i,\eta\in B(x,\pi/2+\delta)$ for any $i$; or
\item $\xi_i,\eta\notin B(x,\pi/2-\delta)$ for any $i$.
\end{enumerate}
Then $\{(\xi_i)'_x\}_{i=1}^k$ is a $(c(\varepsilon),\varkappa(\delta))$-noncritical collection with a regular direction $\eta'_x$.
\end{lem}

In particular, we have

\begin{cor}\label{cor:ind}
Suppose $\xi,\eta,x\in\Sigma$ satisfies $\overline{|\xi\eta|}<\pi/2-\varepsilon$ and $||\xi x|-\pi/2|<\delta$.
Then $\overline{|\xi'_x\eta'_x|}<\pi/2-c(\varepsilon)$ (regardless of whether $|\eta x|\le\pi/2$ or $|\eta x|\ge\pi/2$).
\end{cor}

By induction on dimension, the above lemma implies

\begin{prop}[{\cite[Proposition 4.5]{F:nc}}]\label{prop:dir}
Let $\{\xi_i\}_{i=1}^k$ an $(\varepsilon,\delta)$-noncritical collection of $\Sigma$ with a regular direction $\eta$.
Then
\begin{enumerate}
\item $k\le\dim\Sigma+1$;
\item there exists $v\in\Sigma$ such that
\[|v\xi_1|<\pi/2-c(\varepsilon),\quad|v\xi_i|=\pi/2\,\quad|v\eta|>\pi/2+c(\varepsilon)\]
for any $i\ge 2$.
\end{enumerate}
\end{prop}

Using the above infinitesimal properties, one can prove the following theorem, which generalizes the properties of the strainer maps of Lytchak-Nagano \cite{LN:geod}, \cite{LN:top}.
Recall that a map $f:X\to Y$ between metric spaces is called \textit{$\varepsilon$-open} if $f(B(x,r))$ contains $B(f(x),\varepsilon r)$ for any $x\in X$ and any sufficiently small $r>0$.

\begin{thm}[{\cite[Theorems 1.6, 1.7]{F:nc}}]
Let $f:U\to\mathbb R^k$ be an $(\varepsilon,\delta)$-noncritical map at $p\in U$, where $U$ is a tiny ball in a GCBA space.
Then
\begin{enumerate}
\item $k\le\dim T_p$ and $f$ is a $c(\varepsilon)$-open map near $p$;
\item if $k=\dim T_p$, then $f$ is a bi-Lipschitz open embedding near $p$;
\item if $k<\dim T_p$, then $f$ is a Hurewicz fibration near $p$ with contractible fibers.
\end{enumerate}
\end{thm}

\section{Examples}\label{sec:ex}

In this section we give several examples of extremal subsets in GCBA spaces.
Subsection \ref{sec:pt} provides examples of extremal points.
Subsection \ref{sec:conv} deals with extremal subsets arising from Reshetnyak's gluing theorem.
Subsection \ref{sec:top} proves Theorem \ref{thm:hom}.
Subsection \ref{sec:str} is concerned with the measure-theoretic stratification of Lytchak-Nagano \cite{LN:geod}.

Before giving examples, we define a variant of the notion of an extremal subset.
Recall that a \textit{tripod} in a GCAT($1$) space is a triple of points with pairwise distance $\pi$.
The existence of a tripod in the space of directions of a GCBA space is closely related to the existence of a topological singularity in that space; see \cite[Theorems 1.4, 1.5]{LN:top}.

\begin{dfn}\label{dfn:ext'}
Let $X$ be a GCBA space and $E$ a closed subset.
We say that $E$ is \textit{strongly extremal} if it satisfies the following condition:
\begin{enumerate}
\item[(E')]
any point of $E$ has a neighborhood $U$ in $X$ (contained in some tiny ball) such that if $p\in U\cap E$ is a local minimum point of $\dist_q|_E$, where $q\in U\setminus E$, then $\Sigma_p$ has a tripod containing $q'_p$, i.e., $\diam\Ant(q'_p)=\pi$.
\end{enumerate}
\end{dfn}

Any strongly extremal subset is an extremal subset: the former is defined by the diameter of the set of antipodes, whereas the latter is defined by its circumradius (see \cite[9.2.23]{BBI}).
However, the converse is not always true.

\begin{ex}
Consider a regular tetrahedron $\Delta$ with edgelength $2\pi/3$.
Let $\Sigma$ be the $1$-skeleton of $\Delta$ equipped with the intrinsic metric, which is a GCAT($1$) space.
Let $\{\xi_i\}_{i=1}^4$ be the vertices of $\Sigma$.
For any $\xi\in B(\xi_1,\pi/6)$, we have
\[\min\overline{|\xi\cdot|}>\pi/2,\quad\diam\Ant(\xi)<\pi.\]
Consider the Euclidean cone $K$ over $\Sigma$, which is a GCAT($0$) space.
Let $E$ be the subcone over $\{\xi_2,\xi_3,\xi_4\}$.
Let $p$ be the vertex of $K$ and let $q\in K\setminus E$.
If $q'_p\in B(\xi_1,\pi/6)$, then $\dist_q|_E$ attains a minimum at $p$, and the above estimates imply $\min\overline{|q'_p\cdot|}>\pi/2$, whereas $\diam\Ant(q'_p)<\pi$.
Moreover, $E$ satisfies (strong) extremality for any other $q$.
Therefore $E$ is extremal, but not strongly extremal.
\end{ex}

\subsection{Extremal points}\label{sec:pt}

The first examples are extremal points, that is, one-point extremal subsets.

By definition, a point $p$ in a GCBA space is \textit{extremal} if the antipodal distance between any two directions in $\Sigma_p$ is at least $\pi/2$.
Similarly, $p$ is \textit{strongly extremal} if for any direction in $\Sigma_p$ there exists a tripod containing that direction.
We also say $p$ is \textit{maximally extremal} if the antipodal distance between any two directions is $\pi$.
Any maximally extremal point is a strongly extremal point, but the converse is not always true (e.g., the vertex of the Euclidean cone over a circle of length $3\pi$).
As mentioned in the introduction, extremal points are isolated.

One can easily obtain an extremal point by scaling a GCAT($1$) space and taking a cone over it.
For a metric space $X$ and $\lambda>0$, we denote by $\lambda X$ the space $X$ with the metric multiplied by $\lambda$.

\begin{prop}\label{prop:ver}
Let $\Sigma$ be a compact GCAT($1$) space of dimension $\ge 1$.
For $\lambda\ge1$, consider the Euclidean cone $K$ over $\lambda\Sigma$, which is a GCAT($0$) space.
If $\lambda\ge3/2$, then the vertex of $K$ is a strongly extremal point.
Furthermore, if $\lambda\ge2$, the vertex is a maximally extremal point.
\end{prop}

\begin{proof}
Let us show the first statement.
It suffices to show that for any $\xi\in\lambda\Sigma$, there exists a tripod containing $\xi$.
Since $\dim\Sigma\ge 1$, we may assume $\xi$ is not isolated.
We use the original metric of $\Sigma$.
For any $\eta\in\Sigma$ with distance $2\pi/3$ from $\xi$, we extend the shortest path $\xi\eta$ beyond $\eta$ to a shortest path of length $\pi$ with endpoint $\bar\xi$.
We again extend the shortest path $\eta\bar\xi$ beyond $\bar\xi$ to a shortest path of length $2\pi/3$ with endpoint $\bar\eta$.
Then the triangle inequality implies $|\xi\bar\eta|\ge2\pi/3$.
Hence $\{\xi,\eta,\bar\eta\}$ is a tripod in the rescaled space.

Let us show the second statement.
We use the original metric of $\Sigma$.
It suffices to show that for any $\xi,\eta\in\Sigma$, there exists a point with distance $\ge\pi/2$ from $\xi$ and $\eta$.
Since $\dim\Sigma\ge 1$, we may assume $|\xi\eta|<\pi$.
We extend the shortest path $\xi\eta$ beyond $\eta$ to a shortest path of length $\pi$ with endpoint $\bar\xi$.
If $|\xi\eta|\le\pi/2$, then the claim is trivial.
If $|\xi\eta|>\pi/2$, we again extend the shortest path $\eta\bar\xi$ beyond $\bar\xi$ to a shortest path of length $\pi/2$ with endpoint $\bar\eta$.
Then the triangle inequality implies $|\xi\bar\eta|\ge\pi/2$, which completes the proof.
\end{proof}

We say that a point in a GCBA space is a \textit{branching point} if a sufficiently small punctured ball around the point is disconnected and in addition no neighborhood is homeomorphic to $\mathbb R$.
Note that a small punctured ball is homotopy equivalent to the space of directions at the center; see \cite[Theorem A]{Kr}.

\begin{prop}\label{prop:br}
Any branching point of a GCBA space is maximally extremal.
\end{prop}

\begin{proof}
Let $p$ be a branching point.
We may assume $\dim\Sigma_p\ge1$; otherwise $p$ is maximally extremal by the additional condition in the above definition.
If $p$ is not maximally extremal, then there exist $\xi,\eta\in\Sigma_p$ with $\overline{|\xi\eta|}<\pi$.
This means that $\Sigma_p$ is covered by two contractible balls $B(\xi,\pi)$ and $B(\eta,\pi)$.
Since one can always extend a shortest path to length $\pi$, we see that $\Sigma_p$ is connected.
Together with \cite[Theorem A]{Kr} mentioned above, this contradicts the assumption.
\end{proof}

An \textit{isolated topological singularity} of a GCBA space is a non-manifold point such that a small punctured neighborhood is a topological manifold.
Here a manifold point means a point having a neighborhood homeomorphic to Euclidean space.
For example, any non-manifold point of a CBA homology manifold is an isolated topological singularity, due to Lytchak-Nagano \cite[Theorem 1.2]{LN:top}.
Similarly, an \textit{isolated homological singularity} is a point such that a small punctured neighborhood is a homology manifold, but no neighborhood is a homology manifold.
Note that the space of directions at any isolated homological singularity does not have the homology of a sphere, by \cite[Theorem A]{Kr} and the contractibility of a small metric ball.

\begin{prop}\label{prop:pt}
Any isolated topological or homological singularity of a GCBA space is a maximally extremal point.
\end{prop}

\begin{proof}
First we suppose $p$ is an isolated homological singularity.
We may assume $\dim\Sigma_p\ge1$.
If $p$ is not maximally extremal, there exist $\xi,\eta\in\Sigma_p$ with $\overline{|\xi\eta|}<\pi$, which means that $\Sigma_p$ is covered by two contractible open balls $B(\xi,\pi)$ and $B(\eta,\pi)$.
Since a punctured neighborhood of $p$ is a homology manifold, the same argument as in \cite[Lemma 3.3]{LN:top} shows that $\Sigma_p$ is a homology manifold.
Moreover, the same argument as in \cite[Lemma 8.2]{LN:top} using Poincar\'e duality shows that $\Sigma_p$ is homotopy equivalent to a sphere.
However, as explained above, $\Sigma_p$ does not have the homology of a sphere.
This is a contradiction.

Next we suppose $p$ is an isolated topological singularity.
If $p$ is not maximally extremal, then, as shown above, $\Sigma_p$ is homotopy equivalent to a sphere.
Since a punctured neighborhood of $p$ is a manifold, the topological regularity theorem of Lytchak-Nagano \cite[Theorem 1.1]{LN:top} implies that $p$ is a manifold point.
This is a contradiction.
\end{proof}

\subsection{Convex subsets}\label{sec:conv}

The next examples come from the following Reshetnyak's gluing theorem (\cite{R}).

\begin{thm}\label{thm:glu}
Let $X_1$ and $X_2$ be GCBA spaces with curvature $\le\kappa$.
Suppose $X_1$ and $X_2$ contain proper, closed, locally convex subsets that are isometric to each other.
We denote them by the same notation $C$.
Then the gluing of $X_1$ and $X_2$ along $C$, denoted by $X_1\cup_CX_2$, is a GCBA space with curvature $\le\kappa$.
\end{thm}

See \cite[Theorem 9.1.21]{BBI} or \cite[Theorem 11.1]{BH} for the proof.
Although the original theorem is not concerned with geodesic completeness, it is easy to see that if $X_1$ and $X_2$ are locally geodesically complete, then so is $X_1\cup_CX_2$.
In this subsection, roughly speaking, we show that $C$ itself, and also its complement and boundary are extremal subsets of $X_1\cup_CX_2$.

\subsubsection*{Special case}
First we consider the special case of doubled spaces.
Let $X$ be a GCBA space and $C$ a proper, closed, locally convex subset.
Let $X_1$ and $X_2$ be two copies of $X$.
We denote by $D_C(X)$ the gluing of $X_1$ and $X_2$ along $C$ and call it the \textit{double} of $X$ along $C$.
We regard $X_1$, $X_2$, and $C$ as subsets of $D_C(X)$.
We denote by $\bar C_i^c$ ($i=1,2$) the closure of the complement of $C$ in $X_i$.
We also denote by $\partial C$ the topological boundary of $C$ in $X$, or equivalently in $D_C(X)$.

A metric ball $U$ in $X$ of radius $r$ around a point of $C$ is called a \textit{normal ball} if the concentric closed ball of radius $10r$ is compact and its intersection with $C$ is convex.
We denote by $\tilde U$ the double of $U$ in $D_C(X)$ and call it a \textit{normal ball} in the double.
Then the distance between $q_1,q_2\in\tilde U$ is given by
\[|q_1q_2|=\max_{x\in C}|q_1x|+|q_2x|\]
if $q_1\in X_1$, $q_2\in X_2$; otherwise it is equal to the original distance of $X$.

\begin{prop}\label{prop:conv}
Let $D_C(X)$ be the double as above.
Then
\begin{enumerate}
\item $C$ is a strongly extremal subset of $D_C(X)$.
\item If $C$ has nonempty interior, $\bar C_i^c$ and $\partial C$ are strongly extremal in $D_C(X)$.
\end{enumerate}
\end{prop}

We first summarize the properties of the infinitesimal structure of convex subsets.
Let $X$ and $C$ be as above and $p\in C$.
The \textit{space of direction} $\Sigma_pC$ and the \textit{tangent cone} $T_pC$ are defined in the usual way (see Subsection \ref{sec:gcba}).
It easily follows from the local convexity of $C$ that $T_pC$ is isometric to the limit of $\lambda C$ under the pointed Gromov-Hausdorff convergence $(\lambda X,p)\to(T_p,o)$ as $\lambda\to\infty$.
In particular, $T_pC$ is convex in $T_p$ and hence $\Sigma_pC$ is $\pi$-convex in $\Sigma_p$ (that is, any shortest path between two points of $\Sigma_pC$ with distance less than $\pi$ is contained in $\Sigma_pC$).

To prove Proposition \ref{prop:conv}(2), we need the following observation.

\begin{lem}\label{lem:conv}
Let $X$ and $C$ be as above.
Then for any $p\in\partial C$, there exists $q\notin C$ such that $\dist_q|_C$ attains a minimum at $p$ and also $|qC|$ is locally uniformly bounded away from zero independent of $p$.
In particular, we have
\[\partial\Sigma_pC=\Sigma_p(\partial C).\]
\end{lem}

\begin{proof}
Since $p\in\partial C$, there is a sequence $q_i\in X\setminus C$ converging to $p$.
Let $p_i\in C$ be the unique closest point to $q_i$ (the uniqueness follows from triangle comparison and the local convexity of $C$).
Extend the shortest path $p_iq_i$ beyond $q_i$ to a shortest path of some fixed length (say $\ell$) with endpoint $r_i$.
It follows from triangle comparison and the local convexity of $C$ that $\dist_{r_i}|_C$ also attains a minimum at $p_i$.
Passing to a subsequence, we may assume that $r_i$ converges to the desired point $q$.
The choice of $\ell$ depends only on the radii of a tiny ball and a normal ball around $p$, and hence $|qC|$ is locally uniformly bounded.

Let us show $\partial\Sigma_pC\subset\Sigma_p(\partial C)$.
In fact this inclusion holds for any closed subset $C$.
If $\xi\in\partial\Sigma_pC$, then there is $\eta\in\Sigma_p\setminus\Sigma_pC$ arbitrarily close to $\xi$.
Note that any point close to $p$ on a shortest path starting in the direction $\eta$ is not contained in $C$.
Using this, one can easily find $p_i\in C$ and $q_i\notin C$ both converging to $p$ such that $(p_i)'_p\to\xi$, $(q_i)_p'\to\xi$, and $|pp_i|=|pq_i|$.
Let $r_i\in\partial C$ be the closest point to $q_i$ in $C$.
Then $(r_i)'_p$ also converges to $\xi$ and hence $\xi\in\Sigma_p(\partial C)$.

Let us show $\Sigma_p(\partial C)\subset\partial\Sigma_pC$.
If $\xi\in\Sigma_p(\partial C)$, then there exists $p_i\in\partial C$ converging to $p$ such that $(p_i)'_p$ converges to $\xi$.
By the first statement, for any $\varepsilon>0$, there exists $q_i\notin C$ with $|p_iq_i|/|pp_i|=\varepsilon$ such that $\dist_{q_i}|_C$ attains a minimum at $p_i$.
We may assume that $(q_i)'_p$ converges to $\eta$, which is not contained in $\Sigma_pC$ but is ($\arctan\varepsilon$)-close to $\xi$.
Since $\varepsilon$ is arbitrary, this implies $\xi\in\partial\Sigma_pC$.
\end{proof}

\begin{rem}\label{rem:conv}
The above lemma shows
\begin{align*}
p\in\partial C&\iff\text{$p$ is a minimum point of $\dist_q|_C$ for some $q\notin C$}\\
&\iff\text{$\Sigma_pC=\emptyset$ or $\max|\Sigma_pC\cdot|\ge\pi/2$}\\
&\iff\text{$\Sigma_pC$ is a proper subset of $\Sigma_p$}.
\end{align*}
Indeed, all the implication $\Longrightarrow$ hold and the last statement implies the first one.
\end{rem}

Now let $D_C(X)$ be the double of $X$ along $C$.
For any $p\in C$, we have
\[\Sigma_p(D_C(X))=D_{\Sigma_pC}(\Sigma_pX),\]
where the metric of the right-hand side is the $\pi$-truncated one (see Subsection \ref{sec:gcba}).

\begin{proof}[Proof of Proposition \ref{prop:conv}]
Let us prove (1).
Define a neighborhood $U$ of a point of $\partial C$ in $D_C(X)$ as the intersection of a tiny ball and a normal ball.
Let $p\in U\cap\partial C$ and $q\in U\setminus C$ and suppose $\dist_q|_C$ attains a local minimum at $p$.
We may assume $q\in X_1$.
Replacing $q$ with a point sufficiently close to $p$ on the shortest path $pq$, we may assume that $\dist_q|_C$ attains a global minimum at $p$.
Let $\bar q\in X_2$ be the point corresponding to $q$.
Then the concatenation $qp\bar q$ of the shortest paths $qp$ and $p\bar q$ is a shortest path.
Extend $qp$ beyond $p$ in $X_1$ to a shortest path with endpoint $r\in U$.
Then $rp\bar q$ is also a shortest path.
Hence we obtain a tripod containing $q'_p$.

Let us prove (2).
In view of (1), it suffices to show that $\bar C_i^c$ is strongly extremal.
Define $U$ as above.
Let $p\in U\cap\partial C$ and $q\in U\setminus\bar C_i^c$ and suppose $\dist_q|_{\bar C_i^c}$ attains a local minimum at $p$.
In view of (1), we may assume $q\in C$.
Then $q'_p\in\Sigma_pC$ by the local convexity of $C$.
Furthermore, since $\dist_q|_{\bar C_i^c}$ attains a local minimum at $p$, we have $\partial\Sigma_pC=\emptyset$ or $|q'_p\partial\Sigma_pC|\ge\pi/2$ (note $\partial\Sigma_pC=\Sigma_p(\partial C)$ by Lemma \ref{lem:conv}).
In the former case, $\Sigma_pC$ is open and closed in $\Sigma_p$, and moreover, proper by Remark \ref{rem:conv}.
Hence there exists a tripod containing $q'_p$ in $\Sigma_p(D_C(X))$.
In the latter case, by Remark \ref{rem:conv}, there exist $\xi_i\in\Sigma_pX_i$ ($i=1,2$) such that $|\xi_i\Sigma_pC|\ge\pi/2$.
Then the triple $\{q'_p,\xi_1,\xi_2\}$ is a tripod.
\end{proof}

\subsubsection*{General case}
Next we consider the general case of glued spaces.
Let $X_1$ and $X_2$ be GCBA spaces and suppose they have proper, closed, locally convex subsets that are isometric to each other, denoted by the same notation $C$.
Let $X_1\cup_CX_2$ denote the gluing of $X_1$ and $X_2$ along $C$.
As before, we regard $X_1$, $X_2$, and $C$ as subsets of $X_1\cup_CX_2$, and use the notation $\bar C_i^c$.
One can also define a normal ball in the same way.

Furthermore, we assume that
\begin{equation}\tag{C}\label{eq:conv'}
\Sigma_pC\text{ is proper in }\Sigma_pX_1\iff\Sigma_pC\text{ is proper in }\Sigma_p X_2
\end{equation}
for any $p\in C$.
By Remark \ref{rem:conv}, this is equivalent to the condition that the boundaries of $C$ with respect to $X_1$, $X_2$, and $X_1\cup_CX_2$ coincide.
We denote it by the same notation $\partial C$.

\begin{prop}\label{prop:conv'}
Let $X_1\cup_CX_2$ be the gluing satisfying the assumption \eqref{eq:conv'}.
Then
\begin{enumerate}
\item $C$ is an extremal subset of $X_1\cup_CX_2$.
\item If $C$ has nonempty interior, $\bar C_i^c$ and $\partial C$ are extremal in $X_1\cup_CX_2$.
\item If $C$ has no interior, $C$ is strongly extremal in $X_1\cup_CX_2$.
\end{enumerate}
\end{prop}

We need the following lemma to prove (3).

\begin{lem}\label{lem:conv'}
Let $\Sigma$ be a compact GCAT(1) space of dimension $\ge1$ and $C$ a closed, $\pi$-convex subset with empty interior.
Then there exist $\xi_i\in\Sigma$ ($i=1,2$) with $|\xi_1\xi_2|=\pi$ such that $|\xi_iC|\ge\pi/2$.
\end{lem}

\begin{proof}
The proof is by induction on $\dim\Sigma$.
Note that by Lemma \ref{lem:conv} if $C$ has no interior then so does $\Sigma_pC$ for any $p\in C$.
The case $\dim\Sigma=1$ is straightforward since $C$ is $\pi$-discrete.
In the general case, fix an arbitrary point $p\in C$.
By the inductive assumption there exist $\xi_i\in\Sigma_pC$ ($i=1,2$) with $|\xi_1\xi_2|=\pi$ such that $|\xi_i\Sigma_pC|\ge\pi/2$.
Let $q_i\in\Sigma$ be points in the directions $\xi_i$ with distance $\pi/2$ from $p$, respectively.
Then triangle comparison shows that $|q_1q_2|=\pi$ and $|q_iC|\ge\pi/2$, since $C$ is convex.
\end{proof}

\begin{proof}[Proof of Proposition \ref{prop:conv'}]
To prove (1), we define $U$ as the intersection of a tiny ball and a normal ball as before.
Let $p\in U\cap\partial C$ and $q\in U\setminus C$ and suppose $\dist_q|_C$ attains a local minimum at $p$.
We may assume $q\in X_1$ and $\Sigma_pC\neq\emptyset$.
Then $|q'_p\Sigma_pC|\ge\pi/2$.
It suffices to show that $\overline{|q'_p\xi|}\ge\pi/2$ for any $\xi\in\Sigma_pX_1\cup_{\Sigma_pC}\Sigma_pX_2$.
Note that $\Sigma_pC$ is proper in $\Sigma_pX_2$ by the assumption \eqref{eq:conv'}.

If $\xi\in\Sigma_pX_1$, we choose $\eta\in\Sigma_pX_2$ such that $|\eta\Sigma_pC|\ge\pi/2$ by Remark \ref{rem:conv}.
Then $\eta$ is an antipode for $q'_p$ and $|\xi\eta|\ge\pi/2$.
If $\xi\in\Sigma_pX_2$, we choose an antipode $\eta\in\Sigma_pX_2$ for $\xi$.
Then $|q'_p\eta|\ge\pi/2$.
This completes the proof of (1).

The proof of (2) is the same as that of Proposition \ref{prop:conv}(2), so we omit it.
One can show the existence of a tripod in the case where $q\in C\setminus\partial C$.

Finally we prove (3).
Let $U$, $p$, and $q$ be as in the proof of (1).
By assumption and Lemma \ref{lem:conv}, $\Sigma_pC$ has no interior in $\Sigma_pX_2$.
Since we are assuming $\Sigma_pC\neq\emptyset$, we have $\dim\Sigma_pX_2\ge1$.
Hence by Lemma \ref{lem:conv'} there exist $\xi_i\in\Sigma_pX_2$ ($i=1,2$) with $|\xi_1\xi_2|=\pi$ such that $|\xi_i\Sigma_pC|\ge\pi/2$.
Then $\{q'_p,\xi_1,\xi_2\}$ is a tripod.
\end{proof}

\begin{ques}
Is $C$ strongly extremal in $X_1\cup_CX_2$ when $C$ has nonempty interior?
Note that in view of the proof of Proposition \ref{prop:conv'}(2), if $C$ is strongly extremal, then so are $\bar C_i^c$ and $\partial C$.
\end{ques}

\subsection{Topological singularities}\label{sec:top}

In this subsection we prove Theorem \ref{thm:hom}.

\begin{proof}[Proof of Theorem \ref{thm:hom}]
Suppose $X\setminus X^k_{\hom}$ is closed but not extremal.
Then there are a tiny ball $U$, $p\in U\setminus X^k_{\hom}$, and $q\in U\cap X^k_{\hom}$ such that $\dist_q|_{X\setminus X^k_{\hom}}$ attains a local minimum at $p$ and $\dist_q$ is regular at $p$.
Set $\xi:=q'_p$.

The case $k=1$ is as follows.
Since any point near $p$ on the shortest path $pq$ is contained in $X^1_{\hom}$, we see that $\xi$ is isolated in $\Sigma_p$.
Together with $p\notin X^1_{\hom}$, this implies that $p$ is a branching point (see Proposition \ref{prop:br}).
This is a contradiction.
Thus we assume $k\ge 2$.

We show that $B(\xi,\pi/2)$ in $\Sigma_p$ is a homology ($k-1$)-manifold, or equivalently, the subcone over $B(\xi,\pi/2)$ in $ T_p$ with the vertex removed is a homology $k$-manifold.
This follows from \cite[Lemma 3.3]{LN:top}.
Indeed, since $\dist_q|_{X\setminus X^k_{\hom}}$ attains a local minimum at $p$, the intersection of $B(q,|pq|)$ and a small neighborhood of $p$ is contained in $X^k_{\hom}$, and hence is a homology $k$-manifold (see \cite[Lemma 3.1]{LN:top}; it is actually a topological manifold by Remark \ref{rem:hom}, but we will not use this fact).
Therefore every point of the subcone over $B(\xi,\pi/2)$ except the vertex is contained in a limit of homology manifolds, which is a homology manifold by \cite[Lemma 3.3]{LN:top}.

Now by \cite[Lemma 3.2]{LN:top} based on Poincar\'e duality (see also \cite[Lemma A.1]{N:vol}), we see that for any $0<r<\pi/2$, $\partial B(\xi,r)$ is a homology $(k-2)$-manifold with the same homology as $\mathbb S^{k-2}$.

By assumption, there exists $\eta\in\Sigma_p$ such that $\overline{|\xi\eta|}<\pi/2-\varepsilon$ for some $\varepsilon>0$.
Let $\delta\ll\varepsilon$.
By Corollary \ref{cor:ind}, for any $x\in B(\xi,\pi/2+\delta)\setminus \bar B(\xi,\pi/2-\delta)$, we have $\overline{|\xi'_x\eta'_x|}<\pi/2-c(\varepsilon)$, and hence $|\xi'_x\eta'_x|>\pi/2+c(\varepsilon)$ (see Remark \ref{rem:ineq}).

We cover $\Sigma_p$ by $U=B(\xi,\pi/2+\delta)$ and $V=X\setminus\bar B(\xi,\pi/2-\delta)$.
The ball $U$ is contractible along the shortest paths to $\xi$.
Since $\overline{|\xi\eta|}<\pi/2-\varepsilon$, the complement $V$ is contained in $B(\eta,\pi)$.
In particular, any point of $V$ is joined to $\eta$ by a unique shortest path.
Furthermore, the inequality $|\xi'_x\eta'_x|>\pi/2$ in the previous paragraph shows that such a shortest path remains in $V$.
Hence $V$ is also contractible along the shortest paths to $\eta$.
The same observation shows that $\partial B(\xi,\pi/2-\delta/2)$ is a deformation retract of $U\cap V$.

If $k=2$, the above implies that $\Sigma_p$ is homotopy equivalent to $\mathbb S^1$, which contradicts the assumption $p\notin X^2_{\hom}$.
Suppose $k\ge 3$.
The Mayer-Vietoris exact sequence implies that the reduced homology $\tilde H_\ast(\Sigma_p)$ is equal to $\tilde H_{\ast-1}(\partial B(\xi,\pi/2-\delta/2))$.
Hence $\Sigma_p$ has the same homology as $\mathbb S^{k-1}$.
Since $U\cap V$ is connected for $k\ge3$, the van-Kampen theorem implies that $\Sigma_p$ is simply-connected.
By the Whitehead theorem, $\Sigma_p$ is homotopy equivalent to $\mathbb S^{k-1}$, which contradicts $p\notin X^k_{\hom}$.
\end{proof}

\subsection{Measure-theoretic strata}\label{sec:str}

In this subsection we deal with the measure-theoretic stratification of Lytchak-Nagano \cite{LN:top}.
Note that what we call the measure-theoretic stratification is a dimensional decomposition defined below and is not the rectifiable stratification described in \cite[Theorem 1.6]{LN:geod}.

Let $X$ be a GCBA space and $k$ a positive integer.
We denote by $X^k$ the \textit{$k$-dimensional part} of $X$, that is, the set of points with local dimension $k$, or equivalently, infinitesimal dimension $k$ (see \cite[Theorem 1.2]{LN:geod}).
Set $X^{\ge k}:=\bigcup_{l\ge k}X^l$.
Since the local dimension is upper semicontinuous, $X^{\ge k}$ is closed.

\begin{prop}\label{prop:str}
$X^{\ge k}$ is an extremal subset of $X$.
\end{prop}

\begin{proof}
Suppose $X^{\ge k}$ is not extremal.
Then there are a tiny ball $U$, $p\in U\cap X^{\ge k}$, and $q\in U\setminus X^{\ge k}$ such that $\dist_q|_{X^{\ge k}}$ attains a local minimum at $p$ and $\dist_q$ is regular at $p$.
Set $\xi:=q'_p$.

The case $k=1$ is trivial.
The case $k=2$ follows from the same argument as in the second paragraph of the proof of Theorem \ref{thm:hom}: in this case $p$ is a branching point.
Thus we assume $k\ge 3$.

We show that $B(\xi,\pi/2)$ in $\Sigma_p$ has dimension at most $k-2$.
By assumption, the intersection of $B(q,|pq|)$ and a small neighborhood of $p$ has dimension at most $k-1$.
Together with \cite[Lemma 11.5]{LN:top}, this implies that the subcone over $B(\xi,\pi/2)$ in $T_p$ with the vertex removed has dimension at most $k-1$.
Hence the claim follows.

By assumption, there exists $\eta\in\Sigma_p$ such that $\overline{|\xi\eta|}<\pi/2-\varepsilon$ for some $\varepsilon>0$.
Thus $\Sigma_p$ is covered by $B(\xi,\pi/2)$ and $B(\eta,\pi-\varepsilon)$.
In particular, the annulus $B(\eta,\pi-\varepsilon/2)\setminus B(\eta,\pi-\varepsilon)$ has dimension at most $k-2$.
By geodesic completeness and triangle comparison, one can construct a surjective Lipschitz map from this annulus to $B(\eta,\pi-\varepsilon)$.
Hence the dimension of $B(\eta,\pi-\varepsilon)$ is at most $k-2$.
Therefore we have $\dim\Sigma_p\le k-2$, which is a contradiction.
\end{proof}

Nevertheless, the closure of $X^k$ is not extremal in general.

\begin{ex}\label{ex:str}
Let $S_1$ and $S_2$ be two copies of the unit sphere $\mathbb S^n$, where $n\ge 2$.
Take antipodal points $\xi_1,\eta_1\in S_1$ and $\xi_2,\eta_2\in S_2$.
Let $I$ be the interval of length $<\pi$ with endpoints $\zeta_1$ and $\zeta_2$.
Identify $\xi_1$ with $\xi_2$, $\eta_1$ with $\zeta_1$, and $\eta_2$ with $\zeta_2$.
The resulting space $\Sigma$ is a GCAT($1$) space.
Consider the Euclidean cone $K$ over $\Sigma$.
Then the closure $\bar K^2$ of $K^2$ is the subcone over $I$.
Let $q\in K$ be a point in the direction $\xi_1=\xi_2$.
Then $\dist_q|_{\bar K^2}$ attains a minimum at the vertex $p$ of $K$, but we have $\min\overline{|q'_p\cdot|}<\pi/2$.
\end{ex}

\section{Properties}\label{sec:prop}

In this section we develop the general theory of extremal subsets in GCBA spaces and prove Theorems \ref{thm:dim}, \ref{thm:reg}, and \ref{thm:mfd}.

First we discuss the infinitesimal structure of extremal subsets.

\begin{dfn}\label{dfn:extf}
Let $\Sigma$ be a compact GCAT($1$) space and $F$ a closed subset.
We say that $F$ is \textit{extremal} if it satisfies the following conditions:
\begin{enumerate}
\item[(F1)] if $p\in F$ is a local minimum point of $\dist_q|_F$, where $q\in\Sigma\setminus F$ and $|pq|<\pi$, then it is a critical point of $\dist_q$;
\item[(F2)] for any $q\in\Sigma$ with $|qF|>\pi/2$, we have $\min\overline{|q\cdot|}\ge\pi/2$; if $F=\emptyset$, then we require $\min\overline{|\ ,\ |}\ge\pi/2$.
\end{enumerate}
Note that $\Sigma$ itself is regarded as an extremal subset of $\Sigma$.
\end{dfn}

This definition will be used when considering spaces of directions.
Whenever we use this definition, it will be stated explicitly; otherwise we use Definition \ref{dfn:ext}.

\begin{rem}
The condition (F1) is equivalent to the condition (E) of Definition \ref{dfn:ext} for GCAT($1$) spaces.
The condition (F2) is necessary to prove Proposition \ref{prop:extf} below (see also the next remark).
\end{rem}

\begin{rem}
For extremal subsets in CBB($1$) spaces, the condition corresponding to (F2) is required only when $F$ is empty or one-point (see \cite[Definition 1.1]{PP:ext}).
This is because in the CBB case the condition corresponding to (F1) implies that $\max|F\cdot|\le\pi/2$ as long as $F$ contains at least two points (\cite[Proposition 1.4.1]{PP:ext}).
On the other hand, in the GCBA case, (F1) does not imply (F2).
For example, the subset $I$ of $\Sigma$ in Example \ref{ex:str} satisfies (F1) but not (F2).
\end{rem}

For a closed (not necessarily extremal) subset $E$ of a GCBA space and $p\in E$, we denote by $\Sigma_pE$ the \textit{space of directions} of $E$ at $p$ and by $T_pE$ the \textit{tangent cone} of $E$ at $p$ (see Subsection \ref{sec:gcba}).
Note that they are closed subsets of $\Sigma_p$ and $T_p$, respectively.
The following is a counterpart to \cite[Proposition 1.4]{PP:ext}.

\begin{prop}\label{prop:extf}
Let $E$ be a closed subset of a GCBA space.
Then $E$ is extremal in the sense of Definition \ref{dfn:ext} if and only if $\Sigma_pE$ is extremal in $\Sigma_p$ for any $p\in E$ in the sense of Definition \ref{dfn:extf}.
\end{prop}

\begin{proof}
First we prove the ``if'' part.
Suppose $E$ is not extremal.
Then there are a tiny ball $U$, $p\in U\cap E$, and $q\in U\setminus E$ such that $\dist_q|_E$ attains a local minimum at $p$ and $\dist_q$ is regular at $p$.
By the condition (F2) of Definition \ref{dfn:extf}, we may assume that $\Sigma_pE$ is nonempty.
Setting $\xi:=q'_p$, we have $|\xi\Sigma_pE|\ge\pi/2$ and $\overline{|\xi\eta|}<\pi/2$ for some $\eta\in\Sigma$.
Together with the condition (F2), this implies $|\xi\Sigma_pE|=\pi/2$.
Let $\zeta\in\Sigma_pE$ be a closest point to $\xi$.
Then Corollary \ref{cor:ind} implies that $\overline{|\xi'\eta'|}<\pi/2$ at $\zeta$, which contradicts the extremality of $\Sigma_pE$.

Next we prove the ``only if'' part.
It is clear that $\Sigma_pE$ satisfies the condition (F2) of Definition \ref{dfn:extf}.
Suppose $\Sigma_pE$ does not satisfy the condition (F1).
Then there are $\xi\in\Sigma_pE$ and $\eta\in\Sigma_p\setminus\Sigma_pE$ with $|\xi\eta|<\pi$ such that $\dist_\eta|_{\Sigma_pE}$ attains a local minimum at $\xi$ and $\dist_\eta$ is regular at $\xi$.
Let $\varepsilon>0$ be such that $\min\overline{|\eta'_\xi\cdot|}<\pi/2-\varepsilon$.
We show the following claim.

\begin{clm}\label{clm:extf}
For any $\delta>0$, let $\eta_1$ be a point sufficiently close to $\xi$ on the shortest path $\xi\eta$.
Then we have
\begin{enumerate}
\item $|\xi\eta_1|<\pi/2$;
\item $\dist_{\eta_1}|_{\Sigma_pE}$ attains a (global) minimum at $\xi$;
\item any minimum point of $\dist_{\eta_1}|_{\Sigma_pE}$ is contained in $B(\xi,\delta|\xi\eta_1|)$;
\item $\dist_{\eta_1}$ is regular on $B(\xi,c(\varepsilon)|\xi\eta_1|)$.
\end{enumerate}
\end{clm}

\begin{proof}[Proof of Claim \ref{clm:extf}]
Suppose $\eta_1$ is sufficiently close to $\xi$ on the shortest path $\xi\eta$.
Then (1) is clear and (2) follows from the triangle inequality.

To prove (3), suppose there exist a constant $c>0$ and a sequence $\eta_i\in \xi\eta$ converging to $\xi$ such that a minimum point $\xi_i$ of $\dist_{\eta_i}|_{\Sigma_pE}$ is not contained in $B(\xi,c|\xi\eta_i|)$ (but it is contained in $\bar B(\xi,2|\xi\eta_i|)$).
The triangle inequality shows that the concatenation $\xi_i\eta_i\eta$ of shortest paths is a shortest path.
Consider the pointed Gromov-Hausdorff convergence $(|\xi\eta_i|^{-1}\Sigma_p,\xi)\to(T_\xi,o_\xi)$, where $o_\xi$ is the vertex of $T_\xi$.
Then $\xi\eta$ converges to a ray emanating from $o_\xi$, whereas $\xi_i\eta_i\eta$ converges to a ray emanating from $\lim\xi_i\neq o_\xi$.
However, these two rays coincide beyond $\lim\eta_i$, which contradicts the cone structure of $T_\xi$.

Similarly, (4) is proved by contradiction.
Indeed, the distance function from $\eta'_{\xi}$ is regular on the $c(\varepsilon)$-neighborhood of $o_\xi$ in $T_\xi$ (see \cite[Lemma 5.4]{F:nc}).
Since the regularity is an open condition, one can lift it to a neighborhood of $\xi$.
\end{proof}

Set $v:=\frac{\eta_1}{\cos|\xi\eta_1|}\in T_p$.
Then any minimum point of $\dist_v|_{T_pE}$ is a minimum point of $\dist_{\eta_1}|_{\Sigma_pE}$ and hence the claim implies that $\dist_v$ is regular at such a minimum point.
Take $p_i\in E$ converging to $p$ such that $(p_i)'_p\to\xi$.
Consider the pointed Gromov-Hausdorff convergence $(|pp_i|^{-1}X,p)\to(T_p,o_p)$ and choose $v_i\in|pp_i|^{-1}X$ converging to $v$.
Then $\dist_{v_i}$ is regular at any minimum point of $\dist_{v_i}|_{|pp_i|^{-1}E}$ as it converges to a minimum point of $\dist_v|_{T_pE}$ (note that $\lim|pp_i|^{-1}E$ is contained in $T_pE$; see also Question \ref{ques:cone} below).
This contradicts the extremality of $E$.
\end{proof}

\begin{rem}
The strong version of Definition \ref{dfn:extf} corresponding to Definition \ref{dfn:ext'} is given as follows:
\begin{enumerate}
\item replace ``$p$ is critical for  $\dist_q$'' in (F1) by ``$\Sigma$ has a tripod containing $q'_p$'';
\item replace ``$\min\overline{|q\cdot|}\ge\pi/2$'' and ``$\min\overline{|\ ,\ |}\ge\pi/2$'' in (F2) by ``$\diam\Ant(q)=\pi$'' and  ``$\diam\Ant(\cdot)\equiv\pi$'', respectively.
\end{enumerate}
Then Proposition \ref{prop:extf} can be verified for strongly extremal subsets along almost the same line.
However, we will not use it.
\end{rem}

The proof of the ``only if'' part of Proposition \ref{prop:extf} actually shows that extremal subsets are stable under the Gromov-Hausdorff convergence of GCBA spaces (with a uniform lower bound on the radii of tiny balls).
Here for simplicity we state it only for the convergence of GCAT spaces.

\begin{prop}
Assume that a sequence $X_i$ of GCAT($\kappa$) spaces converges to a GCAT($\kappa$) space $X$ in the Gromov-Hausdorff sense.
Assume further that a sequence $E_i$ of extremal subsets of $X_i$ converges to a closed subset $E$ of $X$ under this convergence.
Then $E$ is an extremal subset of $X$. 
\end{prop}

\begin{proof}
Suppose $E$ is not extremal.
Then there are a tiny ball $U$ in $X$, $p\in U\cap E$, and $q\in U\setminus E$ such that $\dist_q|_E$ attains a local minimum at $p$ and $\dist_q$ is regular at $p$.
As in Claim \ref{clm:extf}, by replacing $q$ with a point sufficiently close to $p$ on the shortest path $pq$, we may assume that $\dist_q$ is regular at any (global) minimum point of $\dist_q|_E$.
Suppose $q_i\in X_i$ converges to $q$.
Since any minimum point of $\dist_{q_i}|_{E_i}$ converges to a minimum point of $\dist_q|_E$, we get a contradiction to the extremality of $E_i$.
\end{proof}

Proposition \ref{prop:extf} also implies

\begin{cor}\label{cor:cone}
Let $E$ be an extremal subset of a GCBA space and $p\in E$.
Then $T_pE$ is an extremal subset of $T_p$.
\end{cor}

\begin{proof}
Let $\Sigma$ be a compact GCAT($1$) space and $F$ an extremal subset in the sense of Definition \ref{dfn:extf}.
Let $S(\Sigma)$ denote the spherical suspension over $\Sigma$ and $S(F)$ the subsuspension over $F$ in $S(\Sigma)$.
Note that the space of directions of $S(F)$ is again a subsuspension except at the poles.
Using this and Proposition \ref{prop:extf}, one can show by induction on dimension that $S(F)$ is an extremal subset of $S(\Sigma)$ in the sense of Definition \ref{dfn:extf} (note that the condition (F2) is trivial since $\max|S(F)\cdot|\le\pi/2$).
Since the space of directions of $T_pE$ is such a subsuspension except at the vertex, the claim follows from Proposition \ref{prop:extf}.
\end{proof}

\begin{ques}\label{ques:cone}
Is the blow-up limit of $E$ at $p$ unique and equal to $T_pE$?
This is true for CBB extremal subsets, but the proof is based on the existence of radial curves (see \cite[Proposition 3.3]{PP:ext} and \cite[p.\ 164]{Pet:semi}).
\end{ques}

Next we observe the behavior of extremal subsets under the set-theoretic operations.
The following is obvious from the definition.

\begin{prop}
The union of two extremal subsets is extremal.
\end{prop}

However, the intersection, and also the closure of the difference of two extremal subsets are not extremal in general.

\begin{ex}[cf.\ \cite{NSY}]\label{ex:nsy}
Let $D$ be the domain on the $xy$-plane bounded by curves $C_1$: $y=x^2$ ($x\ge0$) and $C_2$: $y=-x^2$ ($x\ge 0$) containing the positive $x$-axis.
Let $R_1$ and $R_2$ be two copies of $\mathbb R^2$.
We identify $C_i$ with the positive $x$-axis of $R_i$ ($i=1,2$) and also identify the negative $x$-axes of $R_1$ and $R_2$, which we denote by $C$.
The resulting space $X$ is a GCAT($0$) space.
The subsets $C\cup C_i$ ($i=1,2$) and $C\cup C_1\cup C_2$ are extremal.
However, the intersection of $C\cup C_1$ and $C\cup C_2$, namely $C$, is not extremal.
Similarly, the closure of the difference of $C\cup C_1\cup C_2$ and $C\cup C_i$ is not extremal.
Note that the spaces of directions of these three extremal subsets at the origin coincide, which does not happen for CBB extremal subsets (\cite[Lemma 3.4(2)]{PP:ext}).
See also the next remark.
\end{ex}

\begin{rem}
In the CBB case, the intersection and the closure of the difference of two extremal subsets are extremal (\cite[Proposition 3.5]{PP:ext}).
These facts lead to the stratification of CBB spaces by extremal subsets (\cite[3.8]{PP:ext}).
The proofs are based on the obtuse angle lemma \cite[Lemma 3.1]{PP:ext}, which does not hold for GCBA extremal subsets.
\end{rem}

Now we prove the extremal version of Proposition \ref{prop:dir} (cf.\ \cite[Proposition 1.5]{PP:ext}).
The reader is advised to read the original proof first.

\begin{prop}\label{prop:dirf}
Let $\Sigma$ be a compact GCAT(1) space and $F$ an extremal subset in the sense of Definition \ref{dfn:extf}.
Let $\{\xi_i\}_{i=1}^k$ be an $(\varepsilon,\delta)$-noncritical collection with a regular direction $\eta$.
Then
\begin{enumerate}
\item there exists $u\in F$ such that
\[|u\xi_i|>\pi/2+c(\varepsilon),\quad|u\eta|<\pi/2-c(\varepsilon)\]
for any $i$;
\item there exists $v\in F$ such that
\[|v\xi_1|<\pi/2-c(\varepsilon),\quad|v\xi_i|=\pi/2\,\quad|v\eta|>\pi/2+c(\varepsilon)\]
for any $i\ge 2$.
\end{enumerate}
\end{prop}

\begin{proof}
(1)
Let $u\in F$ be a closest point to $\eta$.
By the condition (F2) of Definition \ref{dfn:extf}, we have $|u\eta|\le\pi/2$ (note also that $F$ is nonempty).
If $|u\xi_i|<\pi/2+\delta$ for some $i$, then by Lemma \ref{lem:ind}(1), we have $\overline{|\eta'\xi_i'|}<\pi/2-c(\varepsilon)$ at $u$, which contradicts the condition (F1) for $\Sigma_uF$.
Since $\delta$ depends only on dimension and $\varepsilon$, we obtain $|u\xi_i|>\pi/2+c(\varepsilon)$ for any $i$.
Similarly if $|u\eta|>\pi/2-\delta$, we get a contradiction by Lemma \ref{lem:ind}(2).
Therefore $|u\eta|<\pi/2-c(\varepsilon)$.

(2)
The proof is almost the same as that of Proposition \ref{prop:dir}(2).
Since the case $k=1$ is included in (1), we assume $k\ge 2$ and hence $\dim\Sigma\ge1$ by Proposition \ref{prop:dir}(1).
We prove the claim by induction on $\dim\Sigma$.

First we set
\[Y:=\left\{x\in F\mid|x\xi_i|\ge\pi/2\ (i\ge 2), \ |x\eta|\ge\pi/2+c(\varepsilon)\right\}\]
and show that $Y$ is nonempty.
We prove that there exists $\zeta\in\Sigma$ such that $\overline{|\zeta\xi_i|}<\pi/2-c(\varepsilon)$ for any $i\ge2$ and $\overline{|\zeta\eta|}<\pi/2-c(\varepsilon)$.
Then a closest point to $\zeta$ in $F$ is contained in $Y$, as shown in (1).
The desired point $\zeta$ is obtained by moving $\xi_1$ toward $\eta$ slightly.
Let $\bar \xi_i$ be an arbitrary antipode of $\xi_i$, where $i\ge 2$.
Then $|\bar\xi_i\eta|<\pi/2-\varepsilon$ and $|\bar\xi_i\xi_1|<\pi/2+\delta$.
Since $|\xi_1\eta|>\pi/2+\varepsilon$ by Remark \ref{rem:ineq}, triangle comparison shows $\angle\bar\xi_i\xi_1\eta<\pi/2-c(\varepsilon)$.
This means that moving $\xi_1$ toward $\eta$ decreases the distance to $\bar\xi_i$ with velocity at least $c(\varepsilon)$.
Hence one can find the desired $\zeta$ in a $c(\varepsilon)$-neighborhood of $\xi_1$.

Next we set
\[Z:=\left\{x\in F\mid |x\xi|=\pi/2\ (i\ge 2),\ |x\eta|\ge\pi/2+c(\varepsilon)\right\}\]
and show that $Z$ is nonempty.
Suppose $x\in Y$.
Then Lemma \ref{lem:ind}(2) implies that $\{\xi_i'\}_{i=2}^k$ is a $(c(\varepsilon),\varkappa(\delta))$-noncritical collection at $x$ with a regular direction $\eta'$.
By the inductive assumption and (1),  we find that the map $f=(|\xi_2\cdot|,\dots,|\xi_k\cdot|)$ restricted to $F$ is $c(\varepsilon)$-open near $x$ (see \cite[2.1.1]{Per:alex} or \cite[Lemma 3.1]{F:nc}).
Furthermore, for any $i\ge 2$ there exists $v_i\in\Sigma_x F$ such that
\[|v_i\xi_i'|<\pi/2-c(\varepsilon),\quad|v_i\xi_j'|=\pi/2\,\quad|v_i\eta'|>\pi/2+c(\varepsilon)\]
for any $j\neq i$ ($\ge2$).
By using the $c(\varepsilon)$-openness of $f|_F$, one can find $y\in F$ near $x$ such that
\[|y\xi_i|<|x\xi_i|,\quad|y\xi_j|=|y\xi_j|,\quad|y\eta|>|y\eta|\]
(see \cite[2.1.3]{Per:alex} or \cite[Lemma 3.2]{F:nc}).
Hence the claim is proved by contradiction.

Let $v\in Z$ be a closest point to $\xi_1$.
Suppose $|v\xi_1|>\pi/2-\delta$.
Then $\{\xi_i'\}_{i=1}^k$ is a $(c(\varepsilon),\varkappa(\delta))$-noncritical collection at $v$.
The same argument as in the previous paragraph shows that there exists a point of $Z$ closer to $\xi_1$ than $v$, which is a contradiction.
\end{proof}

As already used in the proof of (2), the above proposition implies the openness of a noncritical map restricted to an extremal subset.

\begin{prop}\label{prop:open}
Let $X$ be a GCBA space, $E$ an extremal subset, and $p\in E$. 
Suppose $f:X\to\mathbb R^k$ is $(\varepsilon,\delta)$-noncritical at $p$.
Then $f|_E$ is $c(\varepsilon)$-open near $p$.
\end{prop}

For a subset $A\subset X$, we define its \textit{noncritical number} as the maximal integer $k$ such that there exists a noncritical map to $\mathbb R^k$ at some $p\in A$ (as noted in Definition \ref{dfn:nc}, we sometimes omit ``$(\varepsilon,\delta)$'').
We also define its \textit{splitting number} as the maximal integer $k$ such that there exists $p\in A$ for which $T_p$ splits off an $\mathbb R^k$-factor.
It is easy to see that the noncritical number is not less than the splitting number.
The following corollary is an immediate consequence of the above proposition and the result of Lytchak-Nagano \cite[Theorem 1.6]{LN:geod}.

\begin{cor}\label{cor:dim}
The Hausdorff dimension of any open subset $U$ of $E$ is equal to its noncritical number and splitting number.
In particular, it is an integer or infinity.
\end{cor}

\begin{proof}[Proof of Corollary \ref{cor:dim}]
It follows from Proposition \ref{prop:open} that the Hausdorff dimension is not less than the noncritical number.
On the other hand, the result of Lytchak-Nagano \cite[Theorem 1.6]{LN:geod} says that the Hausdorff dimension is not greater than the splitting number.
\end{proof}

From now on, by \textit{dimension} we mean the Hausdorff dimension.
Let us define the \textit{local dimension} $\dim_pE$ by the formula \eqref{eq:dim} and the \textit{$k$-dimensional part} $E^k$ as in the introduction (as noted in Remark \ref{rem:prf}, we will prove Theorem \ref{thm:dim} later).

\begin{cor}\label{cor:dim2}
For any $p\in E$, we have
\[\dim_pE\ge\dim T_pE=\dim\Sigma_pE+1.\]
In particular, if $p\in E^k$, then $\dim T_pE\le k$.
\end{cor}

\begin{proof}
If there exists a noncritical map on $T_pE$, one can lift  it to a noncritical map on $E$.
Thus we have $\dim _pE\ge\dim T_pE$.

If $k$ is the splitting number of $\xi\in\Sigma_pE$, then the splitting number of $\xi$ in $T_p$ is $k+1$.
Thus we have $\dim T_pE\ge\dim\Sigma_pE+1$.

To prove the opposite inequality, suppose $k$ is the splitting number of $v\in T_pE$, where $k\ge 2$.
If $v$ is not the vertex, we may assume $v\in\Sigma_pE$.
It is easy to see that the splitting number of $v$ in $\Sigma_p$ is not less than $k-1$.
If $v$ is the vertex, then there exists a noncritical collection $\{\xi_i\}_{i=1}^k$ of $\Sigma_p$ with a regular direction $\eta$.
By Proposition \ref{prop:dirf}(2), there exists $\zeta\in\Sigma_pE$ such that $|\zeta\xi_i|=\pi/2$ for any $i\ge 2$ and $|\zeta\eta|\ge\pi/2$.
By Lemma \ref{lem:ind}(2) the noncritical number of $\zeta$ is at least $k-1$.
This completes the proof.
\end{proof}

\begin{rem}
The last argument in the above proof shows that $k\le\dim F+1$ in Proposition \ref{prop:dirf}.
\end{rem}

\begin{rem}\label{rem:dim2}
The equality of Corollary \ref{cor:dim2} does not hold in general.
For example, let $X$, $C$, and $D$ be as in Example \ref{ex:nsy}.
Then $C\cup D$ is an extremal subset of $X$ with local dimension $2$ at the origin, but its tangent cone at the origin is $1$-dimensional, i.e., $\mathbb R$.
\end{rem}

Next we prove Theorem \ref{thm:reg}.
Recall that $p\in E$ is a \textit{$k$-regular point} of $E$ if $T_pE$ is isometric to $\mathbb R^k$.

\begin{lem}\label{lem:reg}
Let $X$ be a GCBA space, $E$ an extremal subset, and $p\in E$ with $\dim T_pE\le k$.
If $T_p=\mathbb R^k\times K$, where $K$ is an Euclidean cone with vertex $o$, then $T_pE=\mathbb R^k\times \{o\}$.
In particular, $p$ is a $k$-regular point of $E$.
\end{lem}

\begin{proof}
The splitting number of any point of $\mathbb R^k\times K\setminus\{o\}$ is clearly greater than $k$. 
Since $\dim T_pE\le k$, we have $T_pE\subset \mathbb R^k\times\{o\}$.
If the equality does not hold, then we get a contradiction to the extremality of $T_pE$.
\end{proof}

\begin{proof}[Proof of Theorem \ref{thm:reg}]
By Corollaries \ref{cor:dim}, \ref{cor:dim2}, and Lemma \ref{lem:reg}, we see that the set of $k$-regular points is dense in $E^k$.
Similarly, we find that the splitting number of any non-$k$-regular point in $E^k$ is less than $k$.
Hence by \cite[Theorem 1.6]{LN:geod}, the set of non-$k$-regular points in $E^k$ has Hausdorff dimension at most $k-1$.
The local finiteness of the $k$-dimensional Hausdorff measure on $E^k$ also follows from \cite[Proposition 10.6]{LN:geod}.
The local positiveness follows from the denseness of $k$-regular points and Proposition \ref{prop:open}.
\end{proof}

Finally we prove Theorems \ref{thm:mfd} and \ref{thm:dim}.

\begin{lem}\label{lem:mfd}
Let $X$ be a GCBA space, $E$ an extremal subset, and $p\in E$ with $\dim T_pE\le k$.
Suppose $f:X\to\mathbb R^k$ is $(\varepsilon,\delta)$-noncritical at $p$.
Then
\[\liminf_{E\ni q\to p}\frac{|f(p)f(q)|}{|pq|}\ge c(\varepsilon).\]
\end{lem}

\begin{proof}
Let $f=(|a_1\cdot|,\dots,|a_k\cdot|)$, $a_i\in X$, and $b\in X$ be as in Definition \ref{dfn:nc}.
Suppose $\liminf_{E\ni q\to p}|f(p)f(q)|/|pq|<\delta$.
Fix a sequence $E\ni q_i\to p$ such that $|f(p)f(q_i)/|pq_i|$ converges to the limit inferior.
Passing to a subsequence, we may assume that $(q_i)'_p$ converges to $\xi\in\Sigma_pE$.
The assumption implies $||a_i'\xi|-\pi/2|<\varkappa(\delta)$ (see also \cite[Lemma 5.5]{LN:geod}).
By Lemma \ref{lem:ind}, $\{a_i''\}_{i=1}^k$ is a noncritical collection at $\xi$ with a regular direction $b''$.
Hence we have $\dim\Sigma_pE\ge k$, in contradiction to $\dim T_pE\le k$.
\end{proof}

\begin{proof}[Proof of Theorem \ref{thm:mfd}]
The first statement immediately follows from Corollary \ref{cor:dim} and \cite[Theorem 1.6]{LN:geod}.
Indeed, since the splitting number of $E^k$ is not greater than $k+1$, $E^k$ is $k$-rectifiable by \cite[Theorem 1.6]{LN:geod}.
See also the next remark.

We show the second statement.
Let $p\in E^k$.
By Corollary \ref{cor:dim}, there exist $q\in E$ arbitrarily close to $p$ and an $(\varepsilon,\delta)$-noncritical map $f:U\to\mathbb R^k$, where $U$ is an open neighborhood of $q$ in $E$.
We may assume $U$ is contained in $E^k$.
Then a result of Lytchak \cite[Proposition 1.1]{L}, together with Lemma \ref{lem:mfd}, implies that $U$ contains an open dense subset on which  $f$ is a locally bi-Lipschitz embedding.
This completes the proof.
For the convenience of the reader, we include Lytchak's argument below (from \cite[Lemma 3.1]{L}).

For any positive integer $m$, we define a closed subset $U_m$ in $U$ by
\[U_m:=\left\{x\in U\mid|xy|<1/m,\ y\in U\Rightarrow|f(x)f(y)|\ge c(\varepsilon)|xy|\right\}.\]
Lemma \ref{lem:mfd} implies $U_m$ cover $U$. 
Note that $f$ is locally bi-Lipschitz on each $U_m$.
We also define a locally closed subset $V_m:=U_m\setminus U_{m-1}$.
Then the Baire category theorem implies that the union of the interiors of $V_m$ is dense in $U$ (see \cite[Lemma 2.1]{L}).
This proves the claim.
\end{proof}

\begin{rem}
The first statement also follows from the proof of the second one (cf.\ \cite[Corollary 3.2]{L}).
Indeed, by Theorem \ref{thm:reg}, the set of $k$-regular points in $E^k$ has full measure, and such a regular point has a neighborhood with a noncritical map to $\mathbb R^k$.
By the above argument, the neighborhood is covered by countably many subsets on which the noncritical map is locally bi-Lipschitz.
Therefore $E^k$ is $k$-rectifiable.
\end{rem}

For a subset $A\subset X$, we define its \textit{bi-Lipschitz number} as the maximal integer $k$ such that there exists an open subset of $A$ that is bi-Lipschitz homeomorphic to an Euclidean ball of dimension $k$.

\begin{proof}[Proof of Theorem \ref{thm:dim}]
The Hausdorff dimension is not less than the topological dimension in general.
The topological dimension is not less than the bi-Lipschitz number.
Furthermore, Theorems \ref{thm:reg} and \ref{thm:mfd} imply that the bi-Lipschitz number of any open subset of an extremal subset is not less than the Hausdorff dimension.
This completes the proof.
\end{proof}

\begin{rem}
For the proofs of our results, one can also use the strainer maps of Lytchak-Nagano \cite{LN:geod} instead of noncritical maps.
\end{rem}

\begin{rem}\label{rem:homeo}
In general, a noncritical map (even a strainer map) restricted to an extremal subset is not a homeomorphism even if it has maximal rank.
For example, let $X$, $C_1$, $C_2$, $C$, and $D$ be as in Example \ref{ex:nsy}.
Consider the distance function from a point of $C$ but the origin, which is a strainer map at the origin.
However, its restriction to extremal subsets $C\cup C_1\cup C_2$ and $C\cup D$ are never injective near the origin.
Compare with the CBB case (\cite[Theorem 3.12]{F:reg}).
\end{rem}

\end{document}